\documentclass[a4paper,11pt]{article}
\usepackage{amsmath}
\usepackage{amssymb}
\usepackage{theorem}
\usepackage{pstricks}
\usepackage{euscript}
\usepackage{epic,eepic}
\usepackage{graphicx}
\PassOptionsToPackage{normalem}{ulem}
\newcommand{\menge}[2]{\big\{{#1} \mid {#2}\big\}}

\newcommand{\emp}{\ensuremath{{\varnothing}}}

\newcommand{\scal}[2]{\left\langle{#1}\mid {#2} \right\rangle}

\newcommand{\vuo}{\ensuremath{\mbox{\footnotesize$\square$}}}

\newcommand{\HH}{\ensuremath{\mathcal H}}
\newcommand{\GG}{\ensuremath{\mathcal G}}

\newcommand{\KKK}{\ensuremath{\boldsymbol{\mathcal K}}}
\newcommand{\GGG}{\ensuremath{\boldsymbol{\mathcal G}}}
\newcommand{\VV}{\ensuremath{\boldsymbol{V}}}
\newcommand{\MM}{\ensuremath{\boldsymbol{M}}}
\newcommand{\SSS}{\ensuremath{\boldsymbol{S}}}

\newcommand{\ri}{\ensuremath{\operatorname{ri}}}
\newcommand{\RR}{\ensuremath{\mathbb R}}

\newcommand{\RPP}{\ensuremath{\,\left]0,+\infty\right[}}

\newcommand{\NN}{\ensuremath{\mathbb N}}
\newcommand{\dom}{\ensuremath{\operatorname{dom}}}

\newcommand{\prox}{\ensuremath{\operatorname{prox}}}

\newcommand{\argmin}{\ensuremath{\operatorname{argmin}}}
\newcommand{\ran}{\ensuremath{\operatorname{ran}}}
\newcommand{\zer}{\ensuremath{\operatorname{zer}}}
\newcommand{\gra}{\ensuremath{\operatorname{gra}}}

\newcommand{\vv}{\ensuremath{\boldsymbol{v}}}
\newcommand{\xx}{\ensuremath{\boldsymbol{x}}}
\newcommand{\yy}{\ensuremath{\boldsymbol{y}}}

\newcommand{\bb}{\ensuremath{\boldsymbol{b}}}
\newcommand{\cc}{\ensuremath{\boldsymbol{c}}}
\newcommand{\dd}{\ensuremath{\boldsymbol{d}}}
\newcommand{\aaa}{\ensuremath{\boldsymbol{a}}}
\newcommand{\ww}{\ensuremath{\boldsymbol{w}}}

\newcommand{\TT}{\ensuremath{\boldsymbol{T}}}

\newcommand{\AAA}{\ensuremath{\boldsymbol{A}}}
\newcommand{\BBB}{\ensuremath{\boldsymbol{B}}}
\newcommand{\QQ}{\ensuremath{\boldsymbol{Q}}}

\newcommand{\Id}{\ensuremath{\operatorname{Id}}}

\newcommand{\IId}{\ensuremath{\boldsymbol{\operatorname{Id}}}}
\newcommand{\weakly}{\ensuremath{\rightharpoonup}}

\newtheorem{theorem}{Theorem}[section]

\newtheorem{corollary}[theorem]{Corollary}

\theoremstyle{plain}{\theorembodyfont{\rmfamily}
}
\theoremstyle{plain}{\theorembodyfont{\rmfamily}
}
\theoremstyle{plain}{\theorembodyfont{\rmfamily}
}
\theoremstyle{plain}{\theorembodyfont{\rmfamily}
\newtheorem{example}[theorem]{Example}}
\theoremstyle{plain}{\theorembodyfont{\rmfamily}
\newtheorem{problem}[theorem]{Problem}}
\theoremstyle{plain}{\theorembodyfont{\rmfamily}
\newtheorem{remark}[theorem]{Remark}}
\theoremstyle{plain}{\theorembodyfont{\rmfamily}
}

\definecolor{labelkey}{rgb}{0,0.08,0.45}
\definecolor{refkey}{rgb}{0,0.6,0.0}
\definecolor{Brown}{rgb}{0.45,0.0,0.05}
\definecolor{dgreen}{rgb}{0.00,0.49,0.00}
\definecolor{dblue}{rgb}{0,0.08,0.75}
\RequirePackage[dvips,colorlinks,hyperindex]{hyperref}
\hypersetup{linktocpage=true,citecolor=dblue,linkcolor=dgreen}
\numberwithin{equation}{section}

\begin{document}
\title{\sffamily\huge 
A splitting algorithm for dual monotone inclusions involving 
cocoercive operators\footnote{
This work was supported by the Agence Nationale de la 
Recherche under grant ANR-08-BLAN-0294-02 and the 
Vietnam National Foundation for Science and Technology Development.}}
\author{ B$\grave{\text{\u{a}}}$ng C\^ong V\~u\\[5mm]
\small UPMC Universit\'e Paris 06\\
\small Laboratoire Jacques-Louis Lions -- UMR CNRS 7598\\
\small 75005 Paris, France\\[2mm]
\small\url{vu@ann.jussieu.fr} \\
}
\date{~}
\maketitle
\begin{abstract}
We consider the problem of solving dual monotone inclusions
involving sums of composite parallel-sum type operators.
A feature of this work is to exploit explicitly the cocoercivity of 
some of the operators appearing in the model. Several splitting
algorithms recently proposed in the literature are recovered as
special cases.
\end{abstract}

{\bf Keywords}: 
cocoercivity,
forward-backward algorithm,
composite operator,
duality,
monotone inclusion,
monotone operator,
operator splitting,
primal-dual algorithm

{\bf Mathematics Subject Classifications (2010)} 
47H05, 49M29, 49M27, 90C25 

\section{Introduction}

Monotone operator splitting methods have found many applications in
applied mathematics, e.g., evolution inclusions~\cite{plc2010}, 
partial differential equations~\cite{Atto08,Gabay83,mercier79}, 
mechanics \cite{Glow89},
variational inequalities~\cite{livre1,facchinei03}, Nash 
equilibria \cite{Luis11}, and various optimization problems 
\cite{siop2,cai,Chamb11,jmaa2,Banf09,siam05,JC10,Tseng91}.
In such formulations, cocoercivity often plays a central role; 
see for instance
\cite{plc2010,livre1,Chen97,Siop1,facchinei03,Gabay83,Glow89,%
mercier79,Tseng90,Tseng91,zhu96}. Recall that an operator
$C\colon\HH\to\HH$ is cocoercive with constant $\beta\in\RPP$ if
its inverse is $\beta$-strongly monotone, that is,
\begin{equation}\label{cococo}
(\forall x\in\HH)(\forall y\in\HH)\quad
\scal{x-y}{Cx-Cy}\geq\beta\|Cx-Cy\|^2.
\end{equation}
In this paper, we revisit a general primal-dual splitting framework
proposed in \cite{plc6} in the presence Lipschitzian operators in 
the context of cocoercive operators. This will lead to a
new type of splitting technique and provide a unifying framework for
some algorithms recently proposed in the literature. The problem
under investigation is the following, where the parallel sum
operation is denoted by $\vuo$ (see \eqref{khoiduc}).

\begin{problem}
\label{prob1} 
Let $\HH$ be a real Hilbert space, let $z\in\HH$,
let $m$ be a strictly positive integer,
let $(\omega_i)_{1\leq i \leq m}$ 
be real numbers in $\left]0,1\right]$ such that 
$\sum_{i=1}^m \omega_i =1$,
let $A\colon\HH\to 2^{\HH}$
be maximally monotone, and let $C\colon\HH\to\HH$ 
be $\mu$-cocoercive for some 
$\mu\in\left]0,+\infty\right[$.
For every $i\in\{1,\ldots, m\}$, let $\GG_i$ be a real Hilbert 
space, let $r_i \in \GG_i$, let $B_i\colon \GG_i\to2^{\GG_i}$ 
be maximally monotone, let $D_i\colon \GG_i \to 2^{\GG_i}$
be maximally monotone and
$\nu_i$-strongly monotone for some $\nu_i\in\left]0,+\infty\right[$,
and suppose that $L_i\colon\HH \to\GG_i$ 
is a nonzero bounded linear operator. 
The problem is to solve the primal inclusion
\begin{equation}\label{primal}
\text{find $\overline{x}\in\HH$ such that}\; 
z\in A\overline{x}+\sum_{i=1}^m\omega_iL^{*}_i
\big((B_i\;\vuo\; D_i)(L_i\overline{x}-r_i)\big)+C\overline{x},
\end{equation}
together with the dual inclusion
\begin{equation}\label{dual}
\text{find $\overline{v}_1 \in \GG_1,\ldots, \overline{v}_m \in \GG_m$ such that 
 $(\exists\; x\in\HH)$ }
\begin{cases}
z-\sum_{i=1}^m\omega_i L_{i}^*\overline{v}_i\in Ax + Cx\\
\big(\forall i\in\{1,\ldots,m\}\big)\; \overline{v}_i \in (B_i\;\vuo\; D_i)(L_i x-r_i).
\end{cases}
\end{equation}
We denote by $\mathcal{P}$ and $\mathcal{D}$ the sets of solutions 
to \eqref{primal} and~\eqref{dual}, respectively.
\end{problem}

In the case when $(D_i^{-1})_{1\leq i\leq m}$ and $C$ are general
monotone Lipschitzian operators, Problem~\ref{prob1} was 
investigated in \cite{plc6}. Here are a couple of special cases 
of Problem~\ref{prob1}.

\begin{example}
In Problem~\ref{prob1}, set $z=0$ and
\begin{equation}\label{tytyt}
 \big(\forall i\in\{1,\ldots,m\}\big)\quad B_{i}\colon
v\mapsto \{0\}\quad\text{and}\quad
D_i\colon v\mapsto 
\begin{cases}
\GG_i & \text{if $v=0$},\\
0& \text{if $v \neq 0$}.
\end{cases}
\end{equation}
The primal inclusion~\eqref{primal} reduces to 
\begin{equation}\label{forba}
\text{find $\overline{x}\in\HH$ such that}\; 
0\in A\overline{x}+C\overline{x}.
\end{equation}
This problem is studied in
\cite{plc2010,Chen97,Siop1,siam05,mercier79,Tseng90,Tseng91}. 
\end{example}

\begin{example}\label{LPL}
Suppose that in Problem~\ref{prob1} the operators $(D_i)_{1\leq
i\leq m}$ are as in~\eqref{tytyt}, and that
\begin{equation}
 A\colon
x\mapsto \{0\}\quad\text{and}\quad
C\colon x\mapsto 0.
\end{equation}
Then we obtain the primal-dual pair 
\begin{equation}\label{primalplc}
\text{find $\overline{x}\in\HH$ such that}\; 
z\in \sum_{i=1}^m\omega_iL^{*}_i\big(B_i(L_i\overline{x}-r_i)\big),
\end{equation}
and
\begin{equation}\label{dualplc}
 \text{find $\overline{v}_1 \in \GG_1,\ldots, \overline{v}_m \in \GG_m$ such that 
  }
\begin{cases}
\sum_{i=1}^m\omega_i L_{i}^*\overline{v}_i= z,\\
(\exists\; x\in\HH)\big(\forall i\in\{1,\ldots,m\}\big)\; \overline{v}_i \in B_i(L_i x-r_i).
\end{cases}
\end{equation}
This framework is considered in~\cite{siop2}, where
further special cases will be found. In particular, it contains
the classical Fenchel-Rockafellar \cite{Rock67} and 
Mosco \cite{Mosc72} duality settings, as well as that of
\cite{Atto96}.
\end{example}

The paper is organized as follows.
Section~\ref{recall} is devoted to notation and background.
In Section~\ref{mainresult}, we present our algorithm, prove its
convergence, and compare it to existing work. 
Applications to minimization 
problems are provided in Section~\ref{applications}, where further
connections with the state-of-the-art are made.

\section{Notation and background}\label{recall}
We recall some notation and background from convex analysis and 
monotone operator theory (see \cite{livre1} for a detailed account).

Throughout, $\HH$, $\GG$, and $(\GG_i)_{1\leq i\leq m}$ 
are real Hilbert spaces. The scalars product and the associated 
norms of both $\HH$ and $\GG$ are denoted respectively by 
$\scal{\cdot}{\cdot}$ and $\|\cdot\|$. For every $i \in\{1,\ldots,m\}$,
the scalar product and associated norm of $\GG_i$ are denoted 
respectively by 
$\scal{\cdot}{\cdot}_{\GG_i}$ and $\|\cdot\|_{\GG_i}$.
We denote by $\mathcal{B}(\HH,\GG)$
the space of all bounded linear operators from $\HH$ to $\GG$. 
The symbols $\weakly $ and $\to$ denote respectively 
weak and strong convergence.
Let $A\colon\HH\to 2^{\HH}$ be a set-valued operator.
The domain and the graph of $A$ are respectively defined by 
$\dom A=\menge{x\in\HH}{Ax\neq\emp}$ and 
$\gra A=\menge{(x,u) \in \HH\times\HH}{u\in Ax}$.
We denote by $\zer A=\menge{x\in\HH}{0\in Ax}$ the set of zeros 
of $A$, and by 
$\ran A=\menge{u\in\HH}{(\exists\; x\in\HH)\; u\in Ax}$ 
the range of $A$. The inverse of $A$ is
$A^{-1}\colon\HH\mapsto 2^{\HH}\colon u\mapsto 
\menge{x\in\HH}{u\in Ax}$. The resolvent of $A$ is
\begin{equation}
 J_A=(\Id + A)^{-1}, 
\end{equation}
where $\Id$ denotes the identity operator on $\HH$.
Moreover, $A$ is monotone if 
\begin{equation}
 (\forall(x,y)\in\HH\times \HH)\;
 (\forall (u,v)\in Ax\times Ay)\quad \scal{x-y}{u-v} \geq 0,
\end{equation}
and maximally monotone if it is monotone and there exists no 
monotone operator 
$B\colon\HH\to2^\HH$ such that $\gra B$ properly contains $\gra A$. 
We say that $A$ is uniformly monotone 
at $x\in\dom A$ if there exists an 
increasing function $\phi\colon\left[0,+\infty\right[\to 
\left[0,+\infty\right]$ vanishing only at $0$ such that 
\begin{equation}\label{oioi}
\big(\forall u\in Ax\big)\big(\forall (y,v)\in\gra A\big)
\quad\scal{x-y}{u-v}\geq\phi(\|x-y\|).
\end{equation}
If $A-\alpha\Id$ is monotone for some 
$\alpha\in\left]0,+\infty\right[$, 
then $A$ is said to be $\alpha$-strongly monotone.
The parallel sum of two set-valued operators $A$ and $B$ 
from $\HH$ to $2^{\HH}$ is 
\begin{equation}\label{khoiduc}
A\;\vuo\; B=(A^{-1}+ B^{-1})^{-1}.
\end{equation} 
The class of all lower semicontinuous convex functions 
$f\colon\HH\to\left]-\infty,+\infty\right]$ such 
that $\dom f=\menge{x\in\HH}{f(x) < +\infty}\neq\emp$ 
is denoted by $\Gamma_0(\HH)$. Now, let $f\in\Gamma_0(\HH)$.
The conjugate of $f$ is the function $f^*\in\Gamma_0(\HH)$ defined by
$f^*\colon u\mapsto
\sup_{x\in\HH}(\scal{x}{u} - f(x))$, and the subdifferential 
of $f\in\Gamma_0(\HH)$ is the maximally monotone operator 
\begin{equation}
 \partial f\colon\HH\to 2^{\HH}\colon x
\mapsto\menge{u\in\HH}{(\forall y\in\HH)\quad
\scal{y-x}{u} + f(x) \leq f(y)}
\end{equation} 
with inverse given by
\begin{equation}
(\partial f)^{-1}=\partial f^*.
\end{equation}
Moreover,
the proximity operator of $f$ is
\begin{equation}
\prox_f\colon\HH\to\HH\colon x
\mapsto\underset{y\in\HH}{\argmin}\: f(y) + \frac12\|x-y\|^2.
\end{equation}
We have
\begin{equation}
 J_{\partial f}=\prox_{f}.
\end{equation}
The infimal convolution of two functions $f$ 
and $g$ from $\HH$ to $\left]-\infty,+\infty\right]$ is
\begin{equation}
f\;\vuo\; g\colon\HH\to\left]-\infty,+\infty\right]
\colon x\mapsto\inf_{y\in\HH}(f(x) + g(x-y)).
\end{equation}
Finally, let $S$ be a convex subset of $\HH$. 
The relative interior of $S$, i.e., the set of points $x \in S$ 
such that the cone generated by $−x + S$ is
a vector subspace of $\HH$, is denoted by $\ri S$.

\section{Algorithm and convergence}\label{mainresult}
Our main result is the following theorem, in which we introduce 
our splitting algorithm and prove its convergence. 

\begin{theorem}\label{main*}
In Problem~\ref{prob1}, suppose that
\begin{equation}\label{rang}
z\in\ran
\bigg(A+\sum_{i=1}^m\omega_iL^{*}_i
\big((B_i\;\vuo\; D_i)(L_i\cdot-r_i)\big)+C\bigg).
\end{equation}
Let $\tau$ and $(\sigma_i)_{1\leq i\leq m}$ be strictly positive numbers
such that 
\begin{equation}\label{condke}
2\rho\min\{\mu,\nu_1,\ldots,\nu_m\} > 1,
\text{where $\rho=\min\Big\{\tau^{-1},\sigma^{-1}_1,\ldots,\sigma^{-1}_m\Big\}
\Bigg(1- \sqrt{\tau\sum_{i=1}^m\sigma_i\omega_i\|L_i\|^2}\;\Bigg)$}.
\end{equation}
Let $\varepsilon\in\left]0,1\right[$,
let $(\lambda_n)_{n\in\NN}$ be a sequence in $\left[\varepsilon,1\right]$, 
let $x_0 \in\HH$, let $(a_{1,n})_{n\in\NN}$ and $(a_{2,n})_{n\in\NN}$ 
be absolutely summable sequences in $\HH$. 
For every $i\in\{1,\ldots, m\}$, let $v_{i,0} \in\GG_i$, and 
let $(b_{i,n})_{n\in\NN}$ and $(c_{i,n})_{n\in\NN}$ 
be absolutely summable sequences in $\GG_i$.
Let $(x_n)_{n\in\NN}$ and $(v_{1,n},\ldots,v_{m,n})_{n\in\NN}$ be sequences 
generated by  the following routine
\begin{equation}\label{eqalgo}
(\forall n\in\NN)\quad 
\begin{array}{l}
\left\lfloor
\begin{array}{l}
p_{n}=
J_{\tau A}\Big(x_{n}-\tau\Big(\sum_{i=1}^{m}
\omega_i L_{i}^*v_{i,n}+Cx_{n} + a_{1,n}-z\Big) \Big) + a_{2,n}\\
y_{n}=2p_{n} - x_n\\
x_{n+1}=x_n+\lambda_n(p_{n}-x_n)\\
\operatorname{for}\  i =1, \ldots, m\\
\left\lfloor
\begin{array}{l}
q_{i,n}=
J_{\sigma_i B_{i}^{-1}}\Big(v_{i,n} + 
\sigma_i\Big(L_iy_{n} - D_{i}^{-1}v_{i,n} - c_{i,n} - r_i\Big)\Big)  
+ b_{i,n}\\
v_{i,n+1}=v_{i,n}+\lambda_n(q_{i,n}-v_{i,n}).
\end{array}
\right.\\[2mm]
\end{array}
\right.\\[2mm]
\end{array}
\end{equation}
Then the following hold for some $\overline{x} \in\mathcal{P}$ 
and $(\overline{v}_{1},\ldots,\overline{v}_{m})\in\mathcal{D}$.
\begin{enumerate}
\item \label{jus1}
$x_n\weakly \overline{x}$ and 
$(v_{1,n},\ldots,v_{m,n})
\weakly (\overline{v}_{1},\ldots, \overline{v}_{m})$.
\item \label{jus2}
Suppose that $C$ is uniformly monotone at $\overline{x}$. 
Then $x_n \to \overline{x}$.
\item \label{jus3}
Suppose that $D_{j}^{-1}$ is uniformly monotone at $\overline{v}_{j}$ for
some $j\in\{1,\ldots,m\}$.
Then $v_{j,n} \to \overline{v}_{j}$.
\end{enumerate}
\end{theorem}
\begin{proof}
We define $\GGG$ as the real Hilbert space obtained by 
endowing the Cartesian product $\GG_1\times\ldots\times\GG_m$ 
with the scalar product and the associated norm respectively 
defined by
\begin{equation}
\label{e:palawan-mai2008-}
\scal{\cdot}{\cdot}_{\GGG}
\colon(\vv,\ww)\mapsto
\sum_{i=1}^m\omega_i\scal{v_i}{w_i}_{\GG_i}
\quad\text{and}\quad\|\cdot\|_{\GGG}\colon
\vv\mapsto\sqrt{\sum_{i=1}^m\omega_i\|v_i\|_{\GG_i}^2},
\end{equation}
where $\vv=(v_1,\ldots, v_m)$ and $\ww=(w_1,\ldots, w_m)$
denote generic elements in $\GGG$.
Next, we let $\KKK$ be the Hilbert direct sum
\begin{equation}\label{direct}
\KKK=\HH\oplus\GGG.
\end{equation}
Thus, the scalar product and the norm of $\KKK$ are respectively defined by
\begin{equation}
\label{e:palawan-mai2008a}
\scal{\cdot}{\cdot}_{\KKK}
\colon\big((x,\vv),(y,\ww)\big)\mapsto
\scal{x}{y} + \scal{\vv}{\ww}_{\GGG}
\quad\text{and}\quad\|\cdot\|_{\KKK}\colon
(x,\vv)\mapsto\sqrt{\|x\|^2 + \|\vv\|_{\GGG}^{2}}.
\end{equation}
Let us set
\begin{alignat}{2}\label{maximal1}
\MM\colon\KKK&\to 2^{\KKK}\notag\\
(x,v_1,\ldots,v_m)&\mapsto \big(-z+ Ax\big)
\times\big(r_1 + B_{1}^{-1}v_1\big)\times\ldots\times\big(r_m + B^{-1}_{m}v_m\big).
\end{alignat}
Since the operators $A$ and $(B_i)_{1\leq i\leq m}$ are 
maximally monotone, $\MM$ is maximally
monotone~\cite[Propositions~20.22 and 20.23]{livre1}. 
We also introduce
\begin{alignat}{2}\label{maximal2}
\SSS\colon\KKK&\to \KKK\\
(x,v_1,\ldots,v_m)&\mapsto
\bigg(\sum_{i=1}^m\omega_iL_{i}^*v_i,-L_1x,\ldots,-L_mx\bigg).
\end{alignat}
Note that $\SSS$ is linear, bounded, and skew (i.e, $\SSS^*=-\SSS$).
Hence, $\SSS$ is maximally monotone~\cite[Example~20.30]{livre1}.
Moreover, since $\dom\SSS=\KKK$, $\MM+\SSS$
is maximally monotone~\cite[Corollary~24.24(i)]{livre1}. 
Since, for every $i\in\{1,\ldots,m\}$, $D_i$ is $\nu_i$-strongly monotone,
$D_{i}^{-1}$ is $\nu_i$-cocoercive.  Let us prove that 
\begin{alignat}{2}\label{maximal3}
\QQ\colon\KKK&\to \KKK\notag\notag\\
(x,v_1,\ldots,v_m)&\mapsto\big(Cx,D_{1}^{-1}v_1,\ldots, D_{m}^{-1}v_m\big)
\end{alignat}
is $\beta$-cocoercive with 
\begin{equation}
 \beta=\min\{\mu, \nu_1,\ldots, \nu_m\}.
\end{equation}
For every 
$(x,v_1,\ldots,v_m)$  and every
$(y,w_1,\ldots,w_m)$ in $\KKK$, we have
\begin{alignat}{2}\label{concho}
&\scal{(x,v_1,\ldots,v_m) - (y,w_1,\ldots,w_m)}
{\QQ(x,v_1,\ldots,v_m)-\QQ(y,w_1,\ldots,w_m)}_{\KKK}\notag\\
&= \scal{x-y}{Cx -Cy} + \sum_{i=1}^m\omega_i
\scal{v_i-w_i}{D_{i}^{-1}v_i - D_{i}^{-1}w_i}_{\GG_i}\notag\\
&\geq \mu \|Cx -Cy\|^2 + \sum_{i=1}^m\nu_i\omega_i
\|D_{i}^{-1}v_i - D_{i}^{-1}w_i\|_{\GG_i}^2\notag\\
&\geq \beta\bigg(\|Cx -Cy\|^2 + \sum_{i=1}^m\omega_i
\|D_{i}^{-1}v_i - D_{i}^{-1}w_i\|_{\GG_i}^2\bigg)\notag\\
&=\beta\|\QQ(x,v_1,\ldots,v_m)-\QQ(y,w_1,\ldots,w_m)\|_{\KKK}^2.
\end{alignat}
Therefore, by~\eqref{cococo}, $\QQ$ is $\beta$-cocoercive.
It is shown in~\cite[Eq.~(3.12)]{plc6} that under the condition~\eqref{rang},
$\zer(\MM + \SSS + \QQ)\neq\emp$. Moreover,~\cite[\rm Eq.~(3.21)]{plc6} and 
~\cite[\rm Eq.~(3.22)]{plc6} yield
\begin{equation}\label{khoqua}
(\overline{x},\overline{\vv}) 
\in\zer(\MM + \SSS + \QQ)\Rightarrow \overline{x} \in \mathcal{P}\quad
\text{and}\quad
\overline{\vv}\in\mathcal{D}.
\end{equation}
Now, define
\begin{alignat}{2}\label{vthetare}
\VV\colon\KKK&\to\KKK\notag\notag\\
(x,v_1,\ldots,v_m)&\mapsto
\bigg(\tau^{-1}x-\sum_{i=1}^m\omega_i L^{*}_iv_i,\sigma_{1}^{-1}v_1 -L_1x  
,\ldots,\sigma_{m}^{-1}v_m -L_mx\bigg).
\end{alignat}
Then $\VV$ is self-adjoint. 
Let us check that $\VV$ is $\rho$-strongly positive.
To this end, define
\begin{equation}
\TT\colon\HH\to\GGG\colon 
x\mapsto\Big(\sqrt{\sigma_1}L_1x,\ldots,\sqrt{\sigma_m}L_mx\Big).
\end{equation}
Then, 
\begin{alignat}{2}
(\forall x\in\HH)\quad
\|\TT x\|_{\GGG}^2 
= \sum_{i=1}^m\omega_i \sigma_i\|L_ix\|_{\GG_i}^2
\leq \|x\|^2\sum_{i=1}^m\omega_i \sigma_i\|L_i\|^2,
\end{alignat}
which implies that
\begin{equation}\label{toantu}
 \|\TT\|^2 \leq\sum_{i=1}^m\omega_i \sigma_i\|L_i\|^2.
\end{equation}
Now set
\begin{equation}\label{delke}
 \delta=
\Bigg(\sqrt{\tau\sum_{i= 1}^m \sigma_i\omega_i\|L_i\|^2}\;\Bigg)^{-1} -1.
\end{equation}
Then, it follows from~\eqref{condke} that $\delta > 0$. 
Moreover,~\eqref{toantu} 
and~\eqref{delke} yield 
\begin{equation}\label{pascau}
\tau\|\TT\|^2(1+ \delta)
\leq\tau(1+ \delta)\sum_{i=1}^m\omega_i \sigma_i\|L_i\|^2
= (1+\delta)^{-1}.
\end{equation}
For every $\xx= (x,v_1,\ldots,v_m)$ in $\KKK$, 
by using~\eqref{pascau}, we obtain
\begin{alignat}{2}
\scal{\xx}{\VV \xx}_{\KKK}
&=  \tau^{-1}\|x\|^2 + 
\sum_{i=1}^m \sigma_{i}^{-1}\omega_i\|v_i\|^{2}_{\GG_i}
- 2\sum_{i=1}^m\omega_i\scal{L_ix}{v_i}_{\GG_i}\notag\\
&= \tau^{-1}\|x\|^2 + 
\sum_{i=1}^m \sigma_{i}^{-1}\omega_i\|v_i\|^{2}_{\GG_i}
-2\sum_{i=1}^m\omega_i
\scal{\sqrt{\sigma_i} L_ix}{\sqrt{ \sigma_{i}}^{-1}v_i}_{\GG_i}\notag\\
&=\tau^{-1}\|x\|^2 + 
\sum_{i=1}^m \sigma_{i}^{-1}\omega_i\|v_i\|^{2}_{\GG_i}
- 2\scal{\TT x}{(\sqrt{\sigma_{1}}^{-1}v_1,\ldots,
\sqrt{\sigma_{m}}^{-1}v_m )}_{\GGG}\notag\\
&\geq \tau^{-1}\|x\|^2 + 
\sum_{i=1}^m\sigma_{i}^{-1}\omega_i\|v_i\|^{2}_{\GG_i}
 -\Bigg(\frac{\|\TT x\|_{\GGG}^2}{\tau(1+\delta)\|\TT\|^2}
+ \tau(1+\delta)\|\TT\|^2
\sum_{i=1}^m\sigma_{i}^{-1}\omega_i\|v_i\|^{2}_{\GG_i}\Bigg)\notag\\
&\geq\Big(1-(1+\delta)^{-1}\Big)\bigg(\tau^{-1}\|x\|^2 
+ \sum_{i=1}^m\sigma_{i}^{-1}\omega_i\|v_i\|^{2}_{\GG_i}\bigg)\notag\\
&\geq \Big(1-(1+\delta)^{-1}\Big)
\min\{\tau^{-1},\sigma^{-1}_1,\ldots,\sigma^{-1}_m\}\|\xx\|^{2}_{\KKK}\notag\\
&= \rho\|\xx\|^{2}_{\KKK}\label{27-8-230}.
\end{alignat}
Therefore, $\VV$ is $\rho$-strongly positive. Furthermore,
it follows from~\eqref{27-8-230} that
\begin{equation}\label{conkhi}
 \VV^{-1} \text{ exists and $\|\VV^{-1}\| \leq \rho^{-1}.$ }
\end{equation}
\ref{jus1}: We first observe that~\eqref{eqalgo} is equivalent to
\begin{equation}
\label{rewrite}
(\forall n\in\NN)\quad
\begin{array}{l}
\left\lfloor
\begin{array}{l}
\tau^{-1}(x_n - p_{n}) - 
\sum_{i=1}^{m}\omega_iL_{i}^*v_{i,n}- Cx_n\in \\
\hfill  -z+A(p_{n} -a_{2,n}) +
a_{1,n} - \tau^{-1}a_{2,n}\\
x_{n+1}=x_{n}+\lambda_n(p_{n}-x_n)\\ 
\operatorname{for}\  i =1, \ldots, m\\
\left\lfloor
\begin{array}{l}
\sigma_{i}^{-1}(v_{i,n} - q_{i,n})
- L_i(x_n - p_{n}) -D_{i}^{-1}v_{i,n}\in \\
\hfill\quad\quad\quad\quad\quad\quad r_i+B_{i}^{-1}(q_{i,n} - b_{i,n}) -
 L_ip_{n} + c_{i,n} - \sigma_{i}^{-1}b_{i,n}\\
v_{i,n+1}= v_{i,n}+\lambda_n(q_{i,n}-v_{i,n}).
\end{array}
\right.\\[2mm]
\end{array}
\right.\\[2mm]
\end{array}
\end{equation}
Now set
\begin{equation}
\big(\forall n\in\NN\big)\quad
\begin{cases}
\xx_n=(x_n, v_{1,n},\ldots,v_{m,n})\\
\yy_n=(p_n,q_{1,n},\ldots,q_{m,n})\\
\aaa_{n}=(a_{2,n}, b_{1,n},\ldots, b_{m,n})\\
\cc_n=(a_{1,n}, c_{1,n},\ldots, c_{m,n})\\
\dd_n=(\tau^{-1}a_{2,n},\sigma_{1}^{-1} b_{1,n}
,\ldots, \sigma_{m}^{-1}b_{m,n}).
\end{cases}
\end{equation}
We have
\begin{equation}\label{errorss}
\sum_{n\in\NN}\|\aaa_{n}\|_{\KKK} < +\infty,\quad
\sum_{n\in\NN}\|\cc_n\|_{\KKK} < +\infty,\quad
 \text{and}
\quad 
\sum_{n\in\NN}\|\dd_n\|_{\KKK} < +\infty.
\end{equation}
Furthermore,~\eqref{rewrite} yields
\begin{equation}
(\forall n\in\NN)\quad
\begin{array}{l}
\left\lfloor
\begin{array}{l}
\VV(\xx_n-\yy_{n})-\QQ\xx_n\in 
(\MM +\SSS)(\yy_{n} - \aaa_{n}) + 
\SSS\aaa_{n} + \cc_n - \dd_n\label{4:1:55}\\
\xx_{n+1}=\xx_n+\lambda_n (\yy_{n} - \xx_n).
\end{array}
\right.\\[2mm]
\end{array}
\end{equation}
Next, we set
\begin{equation}\label{gogong}
(\forall n\in\NN)\quad
\;\bb_n=\VV^{-1}\big((\SSS +\VV)\aaa_{n} + \cc_n - \dd_n\big).
\end{equation} 
Then~\eqref{errorss} implies that
\begin{equation}\label{conlon}
 \sum_{n\in\NN}\|\bb_n\|_{\KKK} < +\infty.
\end{equation}
 Moreover, using~\eqref{conkhi} and~\eqref{gogong}, we have
\begin{alignat}{2}
&\;(\forall n\in\NN)\quad\VV(\xx_n-\yy_{n})-\QQ\xx_n\in 
(\MM +\SSS)(\yy_{n} - \aaa_{n}) + 
\SSS\aaa_{n} + \cc_n - \dd_n\label{generic}\notag\\
\Leftrightarrow &\;(\forall n\in\NN)\quad(\VV - \QQ)\xx_{n}\in 
(\MM +\SSS + \VV)(\yy_{n} - \aaa_{n})
 + (\SSS +\VV)\aaa_{n} + \cc_n - \dd_n\notag\\
\Leftrightarrow 
&\;(\forall n\in\NN)\quad\yy_{n} 
=\big(\MM +\SSS + \VV\big)^{-1}\Big((\VV - \QQ)\xx_{n}
-(\SSS +\VV)\aaa_{n} - \cc_n + \dd_n\Big) + \aaa_{n}\notag\\
\Leftrightarrow 
&\;(\forall n\in\NN)\quad\yy_{n}= 
\Big(\IId + \VV^{-1}(\MM +\SSS)\Big)^{-1}\Big(\big(\IId-\VV^{-1}\QQ\big)\xx_{n} 
-\bb_n\Big) + \aaa_{n}.
\end{alignat}
We derive from~\eqref{4:1:55} that 
\begin{alignat}{2}(\forall n\in\NN)\quad
\xx_{n+1}&= \xx_n + \lambda_n\Big(\big(\IId+\VV^{-1}(\MM +\SSS)\big)^{-1}
\big(\xx_{n}-\VV^{-1}\QQ\xx_{n} - \bb_n\big)+\aaa_{n}-\xx_n\Big)\notag\\
&=\xx_n+\lambda_n\Big(J_{\AAA}\big(\xx_{n} - \BBB\xx_{n} - \bb_n\big)
+\aaa_{n}-\xx_n\Big)\label{27:901},
\end{alignat}
where 
\begin{equation}
 \AAA=\VV^{-1}(\MM +\SSS)\quad \text{and}\quad \BBB=\VV^{-1}\QQ.
\end{equation}
Algorithm~\eqref{27:901} has the structure  of the forward-backward 
splitting algorithm~\cite{Siop1}. Hence, it is sufficient 
to check the convergence conditions  of the forward-backward 
splitting algorithm~\cite[Corollary~6.5]{Siop1} to prove our claims.
To this end,
let us introduce the real Hilbert space $\KKK_{\VV}$ with scalar product
and norm defined by
\begin{equation}
\big(\forall (\xx,\yy) \in \KKK\times\KKK\big)\quad\scal{\xx}{\yy}_{\VV} 
= \scal{\xx}{\VV\yy}_{\KKK}\quad \text{and} \quad
 \|\xx\|_{\VV}=\sqrt{\scal{\xx}{\VV\xx}_{\KKK}} , 
\end{equation}
respectively.  
Since $\VV$ is a bounded linear operator, 
it follows from~\eqref{errorss} and~\eqref{conlon}  that
\begin{equation}
\sum_{n\in\NN}\|\aaa_n\|_{\VV} < +\infty\quad \text{and}\quad
 \sum_{n\in\NN}\|\bb_{n}\|_{\VV} < +\infty.
\end{equation}
Moreover, since $\MM +\SSS$ is monotone on $\KKK$, we have
\begin{alignat}{2}
\big(\forall (\xx,\yy) \in \KKK\times\KKK\big)\quad
\scal{\xx-\yy}{\AAA\xx-\AAA\yy}_{\VV} &=
\scal{\xx-\yy}{\VV \AAA \xx-\VV \AAA\yy}_{\KKK} \notag\\
&= \scal{\xx-\yy}{(\MM +\SSS)\xx-(\MM +\SSS)\yy}_{\KKK}\\
&\geq 0.
\end{alignat}
Hence, $\AAA$ is monotone on $\KKK_{\VV}$. Likewise,
$\BBB$ is monotone on $\KKK_{\VV}$. 
Since $\VV$ is strongly positive, and since 
$\MM +\SSS$ is maximally monotone on $\KKK$, 
$\AAA$ is maximally monotone on $\KKK_{\VV}$. 
Next, let us show that $\BBB$ is $(\beta\rho)$-cocoercive on $\KKK_{\VV}$.
Using~\eqref{concho},~\eqref{27-8-230} and~\eqref{conkhi},
we have
\begin{alignat}{2}
 \big(\forall(\xx,\yy) \in \KKK_{\VV}\times \KKK_{\VV}\big)\quad
\scal{\xx - \yy}{ \BBB\xx - \BBB\yy}_{\VV}&=
 \scal{\xx - \yy}{\VV \BBB\xx -\VV \BBB\yy}_{\KKK}\notag\\
&=\scal{\xx - \yy}{\QQ\xx -\QQ \yy}_{\KKK}\notag\\
& \geq \beta \|\QQ\xx -\QQ \yy\|_{\KKK}^2\notag\\
&=\beta \|\QQ\xx -\QQ \yy\|_{\KKK}
\|\QQ\xx -\QQ \yy\|_{\KKK}\notag\\
&= \beta\|\VV^{-1}\|^{-1}\|\VV^{-1}\|
\|\QQ\xx -\QQ \yy\|_{\KKK}\|\QQ\xx -\QQ \yy\|_{\KKK}\notag\\
&\geq \beta\|\VV^{-1}\|^{-1}
\|\VV^{-1}\QQ\xx - \VV^{-1}\QQ \yy\|_{\KKK}
\|\QQ\xx -\QQ \yy\|_{\KKK}\notag\\
&\geq \beta\|\VV^{-1}\|^{-1}
\scal{\VV^{-1}\QQ\xx - \VV^{-1}\QQ \yy}{\QQ\xx -\QQ \yy}_{\KKK}\notag\\
&= \beta\|\VV^{-1}\|^{-1}
\scal{\BBB\xx - \BBB\yy}{\QQ\xx -\QQ \yy}_{\KKK}\notag\\
&= \beta\|\VV^{-1}\|^{-1}
\|\BBB\xx - \BBB \yy\|_{\VV}^2\notag\\
&\geq \beta\rho\|\BBB\xx - \BBB \yy\|_{\VV}^2.
\end{alignat}
Hence, by~\eqref{cococo}, $\BBB$ is $(\beta\rho)$-cocoercive on $\KKK_{\VV}$. 
Moreover, it follows from our assumption that $2\beta\rho > 1$. 
Altogether, by~\cite[Corollary~6.5]{Siop1} the sequence $(\xx_n)_{n\in\NN}$
converges weakly in $\KKK_{\VV}$ to some 
$\overline{\xx} 
= (\overline{x},\overline{v}_{1},\ldots,\overline{v}_{m}) \in\zer(\AAA+ \BBB)=
\zer(\MM+\SSS+\QQ)$.
Since $\VV$ is self-adjoint and $\VV^{-1}$ exists,
the weak convergence of the sequence
$(\xx_n)_{n\in\NN}$ to $\overline{\xx}$ in $\KKK_{\VV}$
is equivalent to the weak convergence of 
$(\xx_n)_{n\in\NN}$ to $\overline{\xx}$ in $\KKK$.
Hence, $\xx_n\weakly\overline{\xx}\in\zer(\MM+\SSS+\QQ)$.
It follows from~\eqref{khoqua} that $\overline{x}\in\mathcal{P}$ and 
$(\overline{v}_{1},\ldots, \overline{v}_{m})\in\mathcal{D}$.
This proves~\ref{jus1}.

\ref{jus2}\&\ref{jus3}: It follows from~\cite[Remark~3.4]{Siop1} that 
\begin{equation}\label{khoqua11}
 \sum_{n\in\NN} \|\BBB \xx_n - \BBB\overline{\xx}\|_{\VV}^2< +\infty.
\end{equation}
On the other hand, from~\eqref{27-8-230} and~\eqref{khoqua11} yield
 $\BBB\xx_n-\BBB\overline{\xx}=\VV^{-1}(\QQ\xx_n - \QQ\overline{\xx}) \to 0$,
which implies that $\QQ\xx_n-\QQ\overline{\xx}\to 0$. Hence, 
\begin{equation}\label{1:04:19}
Cx_n\to C\overline{x} \quad\text{and}\quad 
\big(\forall i\in\{1,\ldots,m\}\big)\quad
D_{i}^{-1}v_{i,n}\to D_{i}^{-1}\overline{v}_{i}. 
\end{equation}
If $C$ is uniformly monotone at $\overline{x}$, 
then there exists an increasing function
$\phi_C\colon\left[0,+\infty\right[\to\left[0,+\infty\right]$ 
vanishing only at $0$ such that 
\begin{equation}\label{7:4:59}
 \phi_C(\|x_n-\overline{x}\|)\leq\scal{x_n-\overline{x}}{Cx_n-C\overline{x}}
\leq\|x_n -\overline{x}\|\;\|Cx_n-C\overline{x}\|.
\end{equation}
Notice that $(x_n-\overline{x})_{n\in\NN}$ is bounded. 
It follows from~\eqref{1:04:19} and~\eqref{7:4:59} that 
$x_n\to\overline{x}$. This proves~\ref{jus2}, 
and~\ref{jus3} is proved in a similar fashion. 
\end{proof}

\begin{remark} Here are some remarks concerning the 
connections between our framework and existing work.
\begin{enumerate}
\item
The strategy used in the proof of Theorem~\ref{main*}\ref{jus1} 
is to reformulate algorithm~\eqref{eqalgo} as a forward-backward 
splitting algorithm in a real Hilbert space endowed with a suitable 
norm. This renorming technique was used in~\cite{He10} for a 
minimization problem in finite-dimensional spaces. The same 
technique is also used in the primal-dual minimization problem of 
\cite{Condat}. 
\item
Consider the special case when $z=0$, and
$(B_i)_{1\leq i\leq m}$ and $(D_i)_{1\leq i\leq m}$
are as in~\eqref{tytyt}. Then algorithm~\eqref{eqalgo} reduces to
\begin{equation}
(\forall n\in\NN)\quad
x_{n+1}=x_n + \lambda_n\bigg( 
J_{\tau A}\Big(x_{n}-\tau(Cx_{n} + a_{1,n}) \Big) + a_{2,n} -x_n\bigg),
\end{equation}
which is the standard forward-backward splitting 
algorithm~\cite[Algorithm~6.4]{Siop1} where
the sequence $(\gamma_n)_{n\in\NN}$ in~\cite[Eq. (6.3)]{Siop1} 
is constant. 
\item 
The inclusions~\eqref{primalplc} and~\eqref{dualplc} in
Example~\ref{LPL} can be solved by~\cite[Theorem~3.8]{siop2}. 
However, the algorithm resulting from \eqref{eqalgo} in this special 
case is different from that of \cite[Theorem~3.8]{siop2}.
\item 
In Problem~\ref{prob1}, since $C$ and $(D_{i}^{-1})_{1\leq i\leq m}$ 
are cocoercive, they are Lipschitzian. Hence, Problem~\ref{prob1} 
can be solved by the algorithm proposed in~\cite[Theorem~3.1]{plc6},
which has a different structure from the present algorithm.
\item 
Consider the special case when $z=0$ and 
$(\forall i \in\{1,\ldots,m\})\; \GG_i =\HH, L_i=\Id$, 
$D_{i}^{-1}=0, r_i =0$. Then the primal inclusion~\eqref{primal} 
reduces to
\begin{equation}\label{primra}
\text{find $\overline{x}\in\HH$ such that}\; 
0\in A\overline{x}+\sum_{i=1}^m\omega_iB_i\overline{x} + C\overline{x}.
\end{equation}
This inclusion can be solved by the algorithm proposed 
in~\cite{Raguet11}, which is not designed as a primal-dual scheme.
\end{enumerate}
\end{remark}

\section{Application to minimization problems}\label{applications}
We provide an application of the algorithm~\eqref{eqalgo} 
to minimization problems, by revisiting
\rm~\cite[Problem~4.1]{plc6}.

\begin{problem}\label{App1}
Let $\HH$ be a real Hilbert space, let $z\in\HH$,
let $m$ be a strictly positive integer,
let $(\omega_i)_{1\leq i \leq m}$ 
be real numbers in $\left]0,1\right]$ such that 
$\sum_{i=1}^m \omega_i =1$,
let $f\in\Gamma_0(\HH)$, and let $h\colon\HH\to\RR$ be convex and
differentiable with a $\mu^{-1}$-Lipschitzian gradient for 
some $\mu\in\left]0,+\infty\right[$.
For every $i \in \{1,\ldots, m\}$, let $\GG_i$ be a real Hilbert space,
let $r_i \in\GG_i$, let $g_i \in \Gamma_0(\GG_i)$, 
let $\ell_i \in \Gamma_0(\GG_i)$ be $\nu_i$-strongly convex, for some
$\nu_i\in\left]0,+\infty\right[$, and
suppose that  $L_i\colon\HH\to\GG_i$ is a nonzero bounded linear operator.
Consider the primal problem
\begin{equation}\label{primalex6}
\underset{x\in\HH}{\text{minimize}} \;
f(x)+\sum_{i=1}^m\omega_i(g_i\;\vuo\;\ell_i)(L_ix-r_i) + h(x) -\scal{x}{z},
\end{equation}
and the dual problem  
\begin{equation}\label{dualex6}
 \underset{v_1\in\GG_1,\ldots,v_m\in\GG_m}{\text{minimize}} \;
(f^*\;\vuo\; h^*)\bigg(z-\sum_{i=1}^m\omega_i L_{i}^*v_i\bigg) +
\sum_{i=1}^m \omega_i\big(g^{*}_i(v_i)+\ell^{*}_i(v_i) +\scal{v_i}{r_i}_{\GG_i}\big).
\end{equation}
We denote by $\mathcal{P}_1$ and $\mathcal{D}_1$ the sets of 
solutions to~\eqref{primalex6} and~\eqref{dualex6}, respectively.
\end{problem}
\begin{corollary}\label{probex6}
In Problem~\ref{App1}, suppose that 
\begin{equation}\label{ranex6}
z\in\ran\bigg( \partial f + 
\sum_{i=1}^m\omega_iL^{*}_i
\big((\partial g_i\;\vuo\;\partial\ell_i)(L_i\cdot-r_i)\big)
+\nabla h\bigg).
\end{equation}
Let $\tau$ and $(\sigma_i)_{1\leq i\leq m}$ be strictly positive numbers
such that 
\begin{equation}\label{siez}
2\rho\min\{\mu,\nu_1,\ldots,\nu_m\} > 1,
\text{where $\rho 
= \min\Big\{\tau^{-1},\sigma^{-1}_1,\ldots,\sigma^{-1}_m\Big\}
\Bigg(1- \sqrt{\tau\sum_{i=1}^m\sigma_i\omega_i\|L_i\|^2}\Bigg)$}.
\end{equation}
Let $\varepsilon\in\left]0,1\right[$
and let $(\lambda_n)_{n\in\NN}$ be a sequence 
in $\left[\varepsilon,1\right]$, let $x_0 \in\HH$, 
let $(a_{1,n})_{n\in\NN}$ and $(a_{2,n})_{n\in\NN}$ 
be absolutely summable sequences in $\HH$. 
For every $i\in\{1,\ldots, m\}$, let $v_{i,0} \in\GG_i$, 
and let $(b_{i,n})_{n\in\NN}$ and $(c_{i,n})_{n\in\NN}$ 
be absolutely summable sequences in $\GG_i$.
Let $(x_n)_{n\in\NN}$ and $(v_{1,n},\ldots,v_{m,n})_{n\in\NN}$ 
be sequences generated by  the following routine
\begin{equation}\label{eqalgoex6}
(\forall n\in\NN)\quad 
\begin{array}{l}
\left\lfloor
\begin{array}{l}
p_{n}=
\prox_{\tau f}\Big(x_{n} -\tau\Big(\sum_{i=1}^{m}
\omega_i L_{i}^*v_{i,n}+\nabla h(x_{n}) + a_{1,n}-z\Big)\Big)+a_{2,n}\\
y_{n} =2p_{n} -x_n\\
x_{n+1}=x_n + \lambda_n(p_{n}-x_n)\\
\operatorname{for}\  i =1, \ldots, m\\
\left\lfloor
\begin{array}{l}
q_{i,n}=
\prox_{\sigma_i g_{i}^{*}}\Big(v_{i,n} + 
\sigma_i\Big(L_iy_{n}-\nabla\ell^{*}_i(v_{i,n}) + c_{i,n}
- r_i\Big)\Big)+b_{i,n}\\
v_{i,n+1}=v_{i,n}+\lambda_n(q_{i,n}-v_{i,n}).
\end{array}
\right.\\[2mm]
\end{array}
\right.\\[2mm]
\end{array}
\end{equation}
Then the following hold for some $\overline{x}\in\mathcal{P}_1$ and 
$(\overline{v}_{1},\ldots,\overline{v}_{m})\in\mathcal{D}_1$.
\begin{enumerate}
\item\label{7:10:05}
$x_n\weakly \overline{x}$ and 
$(v_{1,n},\ldots,v_{m,n}) \weakly (\overline{v}_{1},\ldots,\overline{v}_{m})
$.
\item\label{7:10:06}
Suppose that $h$ is uniformly convex at $\overline{x}$. 
Then $x_n \to \overline{x}$.
\item\label{7:10:07}
Suppose that $\ell_{j}^{*}$ is uniformly convex at $\overline{v}_{j}$
for some $j\in\{1,\ldots,m\}$. Then $v_{j,n}\to\overline{v}_{j}$.
\end{enumerate}
\end{corollary}
\begin{proof}
The connection between Problem~\ref{App1} and Problem~\ref{prob1} is 
established in the proof of~\cite[Theorem~4.2]{plc6}. 
Since $\nabla h$ is $\mu^{-1}$-Lipschitz continuous, 
by the Baillon-Haddad Theorem~\cite{Baillon77,plc-hh}, 
it is $\mu$-cocoercive. Moreover since, 
for every $i\in\{1,\ldots,m\}$, $\ell_{i}$ is $\nu_i$-strongly 
convex, $\partial \ell_{i}$ is $\nu_{i}$-strongly monotone. 
Hence, by applying Theorem~\ref{main*}\ref{jus1} 
with $A=\partial f$, $J_{\tau A}=\prox_{\tau f}$, 
$C=\nabla h$ and for every $i \in\{1,\ldots,m\}$, 
$D_{i}^{-1}=\nabla \ell^{*}_i$, $B_i=\partial g_i$, 
$J_{\sigma_i B^{-1}_i}=\prox_{\sigma_{i}g_{i}^*}$, we obtain that
the sequence $(x_n)_{n\in\NN}$ converges weakly to some
$\overline{x}\in\HH$ such that 
\begin{equation}
z\in \partial f(\overline{x}) + 
\sum_{i=1}^m\omega_iL^{*}_i
\big((\partial g_i\;\vuo\;\partial\ell_i)(L_i\overline{x}-r_i)\big)
+\nabla h(\overline{x}),
\end{equation}
and the sequence $((v_{1,n},\ldots,v_{m,n}))_{n\in\NN}$ 
converges weakly to some
$(\overline{v}_{1},\ldots,\overline{v}_{m})$ such that 
\begin{equation}
\big(\exists\; x\in\HH\big)\quad
\begin{cases}
z-\sum_{i=1}^m\omega_i L_{i}^*\overline{v}_i 
 \in \partial f(x)+ \nabla h(x)\\
(\forall i\in\{1,\ldots,m\})\quad\overline{v}_i
 \in (\partial g_i\;\vuo\; \partial \ell_i)(L_ix-r_i).
\end{cases}
\end{equation}
As shown in the proof of~\cite[Theorem~4.2]{plc6},
$\overline{x}\in\mathcal{P}_1$ and 
$(\overline{v}_{1},\ldots,\overline{v}_{m})\in\mathcal{D}_1$. 
This proves~\ref{7:10:05}. 
Now, if $h$ is uniformly convex at $\overline{x}$, 
then $\nabla h$ is uniformly monotone at $\overline{x}$. 
Hence,~\ref{7:10:06} follows from 
Theorem~\ref{main*}\ref{jus2}. Similarly,~\ref{7:10:07} follows from 
Theorem~\ref{main*}\ref{jus3}. 
\end{proof}

\begin{remark} Here are some observations on the above results.
\begin{enumerate}
\item 
If a function $\varphi\colon\HH\to\RR$ is convex and
differentiable function with a $\beta^{-1}$-Lipschitzian 
gradient, then $\nabla\varphi$ is 
$\beta$-cocoercive~\cite{Baillon77,plc-hh}.
Hence, in the context of convex minimization problems, the
restriction of cocoercivity made in Problem~\ref{prob1} with respect
to the problem considered in \cite{plc6} disappears. Yet, the
algorithm we obtain is quite different from that 
proposed in \cite[Theorem~4.2]{plc6}. 
\item 
Sufficient conditions which ensure that~\eqref{ranex6}
is satisfied are provided in~\cite[Proposition~4.3]{plc6}.
For instance, if \eqref{primalex6} has at 
least one solution, and if $\HH$ and $(\GG_i)_{1\leq i\leq m}$ 
are finite-dimensional, and there exists $x\in\ri\dom f$ such that
\begin{equation}
\big(\forall i\in\{1,\ldots,m\}\big)\quad
L_i x-r_i \in \ri\dom g_i + \ri\dom\ell_i,
\end{equation}
then~\eqref{ranex6} holds.
\item 
Consider the special case when $z=0$ and, for every 
$i\in\{1,\ldots,m\}$,
$r_i=0, \sigma_i=\sigma\in\left]0,+\infty\right[$, and
\begin{equation}\label{indic}
 \ell_i\colon v\mapsto 
\begin{cases}
 0&\text{if $v=0$,}\\
+\infty&\text{otherwise}.
\end{cases}
\end{equation}
Then, \eqref{eqalgoex6} reduces to
\begin{equation}\label{eqalgoex611}
(\forall n\in\NN)\quad 
\begin{array}{l}
\left\lfloor
\begin{array}{l}
p_{n}=
\prox_{\tau f}\Big(x_{n} -\tau\Big(\sum_{i=1}^{m}
\omega_i L_{i}^*v_{i,n}+\nabla h(x_{n}) + a_{1,n}\Big)\Big)+a_{2,n}\\
y_{n} =2p_{n} -x_n\\
x_{n+1}=x_n + \lambda_n(p_{n}-x_n)\\
\operatorname{for}\  i =1, \ldots, m\\
\left\lfloor
\begin{array}{l}
q_{i,n}=
\prox_{\sigma g_{i}^{*}}\Big(v_{i,n} + 
\sigma\big(L_iy_{n}+ c_{i,n}\big)\Big)+b_{i,n}\\
v_{i,n+1}=v_{i,n}+\lambda_n(q_{i,n}-v_{i,n}),
\end{array}
\right.\\[2mm]
\end{array}
\right.\\[2mm]
\end{array}
\end{equation}
which is
the method proposed in~\cite[Eq.~(36)]{Condat}. 
However, in this setting, the conditions~\eqref{siez}
and~\eqref{ranex6} are different 
from the conditions~\cite[Eq.~(38)]{Condat}
and~\cite[Eq.~(39)]{Condat}, respectively.
Moreover, the present paper provides the strong convergence
conditions.
\item 
In finite-dimensional spaces, with exact implementation of the
operators, and with the further restriction that
$m =1$, $h\colon x\mapsto 0$, 
$\ell_1$ is as in~\eqref{indic}, 
$r_1=0$, and $z=0$, \eqref{eqalgoex6}
remains convergent if $\lambda_n\equiv\lambda\in\left]0,2\right[$
under the same condition presented here~\cite[Remark~5.4]{He10}.
If we further impose the restriction $\lambda_n\equiv 1$, 
then \eqref{eqalgoex6} reduces to the method proposed 
in~\cite[Algorithm~1]{Chamb11}. 
An alternative primal-dual algorithm for this problem
is proposed in~\cite{chen}.
\end{enumerate}
\end{remark}

\noindent{{\bfseries Acknowledgement.}}
I thank Professor Patrick L. Combettes for bringing this problem 
to my attention and for helpful discussions.

\end{document}